\documentclass[12pt,reqno]{amsart}
\usepackage[top=1in, bottom=0.8in, left=1in, right=1in]{geometry}
\usepackage{times}

\usepackage[usenames,dvipsnames]{color}
\usepackage{tikz}
\usepackage{ifthen}
\usepackage{tikz-3dplot}
\usepackage{pgfplots}
\usepackage{amssymb}
\usepackage{bbm}
\usepackage{wasysym}
\usepackage{enumitem}
\usepackage{amsthm}

\numberwithin{equation}{section}

\newtheorem{theorem}{Theorem}[section]
\newtheorem{proposition}[theorem]{Proposition}
\newtheorem{lemma}[theorem]{Lemma}
\newtheorem{corollary}[theorem]{Corollary}

\theoremstyle{definition}

\newtheorem{definition}[theorem]{Definition}
\newtheorem{example}[theorem]{Example}

\newtheorem{remark}[theorem]{Remark}

\newcommand\<{\langle}

\newcommand\NN{{\mathbb N}}

\newcommand\ZZ{{\mathbb Z}}

\newcommand\QQ{{\mathbb Q}}

\newcommand{\bDelta}{{\Delta^c}}
\newcommand{\hull}{\operatorname{Hull}}

\newcommand{\Sat}{{\operatorname{Sat}}}

\newcommand\kk{{\mathbbm{k}}}

\newcommand\rank{{\rm rank}}

\newcommand\chr{{\rm char}}
\newcommand\Memb{M_{\rm emb}}

\newcommand{\chark}{{\rm char (\kk)}}

\def\endrk{\hfill$\hexagon$}

\newcommand*{\defeq}{\mathrel{\vcenter{\baselineskip0.5ex \lineskiplimit0pt
                     \hbox{\scriptsize.}\hbox{\scriptsize.}}}%
                     =}

\newcommand\minus{\smallsetminus}

\renewcommand\>{\rangle}

\begin{document}
\mbox{}
\vspace*{-12mm}
\title{Decompositions of cellular binomial ideals} 

\author{Zekiye Sahin Eser}
\address{Department of Mathematics \\
Texas A\&M University \\ College Station, USA.}
\curraddr{Department of Mathematics Education \\ Zirve University \\
  Gaziantep, Turkey.}
\email{zekiye.eser@zirve.edu.tr}

\author{Laura Felicia Matusevich}
\address{Department of Mathematics \\
Texas A\&M University \\ College Station, USA.}
\email{laura@math.tamu.edu}

\thanks{
The authors were partially supported by NSF Grant DMS 1001763.
}


\begin{abstract}
Without any restrictions on the base field, we compute the hull
and prove a conjecture of Eisenbud and Sturmfels giving an unmixed
decomposition of a cellular binomial 
ideal. 
Over an algebraically closed field, we further obtain an explicit (but not
necessarily minimal) primary decomposition of such an ideal.
\end{abstract}
\maketitle

\vspace*{-12mm}
\section{Introduction}
\label{sec:introduction}

%

A \emph{binomial} is a polynomial with at most two terms; a \emph{binomial
ideal} is an ideal generated by binomials. 
From an algebro-geometric point of view, the varieties associated to
binomial ideals are unions of (translates of) toric
varieties~\cite{ES}, and
thus of much interest. Combinatorially, binomial ideals already
contain the class of monomial ideals, a cornerstone of
combinatorial commutative algebra~\cite{MS, St}. 
Moreover, equations with two terms are used in a variety of applied
contexts where the primary decompositions of the corresponding
binomial ideals carry important information (see for
example~\cite{DSS, HHHKR, SS, GM} and the survey~\cite{M}).

%
%

The systematic study of the primary decomposition of binomial ideals
was initiated by Eisenbud and Sturmfels in~\cite{ES}. Over an
algebraically closed field, they proved that the associated primes and
primary components of binomial ideals can be chosen
binomial. These results, while constructive, are not very explicit.
%
%
%
%
Combinatorial descriptions of the primary components of a binomial
ideal were provided by Dickenstein, Matusevich and Miller~\cite{DMM},
in terms of connected components of graphs with infinitely many
vertices, under the assumption that the base field is algebraically closed of
characteristic zero. Later
Kahle and Miller~\cite{KM} removed all base field assumptions, and
introduced new decompositions of binomial ideals, 
called \emph{mesoprimary decompositions}, that are built
to capture the combinatorial structure of a binomial ideal, and from which
binomial primary decompositions can be recovered.
The results in~\cite{DMM} and~\cite{KM} 
are important and useful for theoretical purposes,
but are currently computationally infeasible.

The goal of this article is to strike a balance between the
computational and combinatorial points of view, by providing binomial
primary decompositions that are algorithmically computable but significantly more
explicit than those in~\cite{ES}.
To do this, we concentrate on \emph{cellular binomial ideals}, modulo
which every variable is either nilpotent or a nonzerodivisor.
We remark that binomial primary decomposition can be reduced to the cellular case, as
any binomial ideal has a binomial cellular
decomposition~\cite[Theorem~6.2 or Proposition~7.2]{ES}.

Our first main result, Theorem~\ref{thm:cellularPrimDecV2}, is a
description of the minimal primary components of a cellular binomial
ideal over an algebraically closed field. We explicitly give
all monomials in the desired
components; the binomials are also described, but only
up to saturation of the nonzerodivisor variables. The special case
when the cellular binomial ideal 
has a unique minimal prime and the base field is algebraically closed
of characteristic zero had already been proved by Kahle~\cite{KM2}.

Our second main result, Theorem~\ref{thm:cellularHull}, is an
expression for the \emph{hull} (the intersection of the minimal
primary components) of a cellular binomial 
ideal $I$ as a sum $I + M$, where $M$ is an explicit monomial ideal
(in fact, the same monomial ideal appearing in
Theorem~\ref{thm:cellularPrimDecV2}). We can think of
Theorem~\ref{thm:cellularHull} as an alternative proof of the fact
that the hull of a cellular binomial ideal is
binomial, a key result from~\cite{ES}. We remark that the hull of a
general binomial ideal need not be binomial~\cite{MO}.

Our final main result, Theorem~\ref{thm:unmixedDecomposition},
is a proof of a conjecture of Eisenbud and
Sturmfels giving a decomposition of
a cellular binomial ideal as a finite intersection of unmixed cellular
binomial ideals. Corollary~8.2 in~\cite{ES} gives this result over a
field of characteristic zero;
Theorem~\ref{thm:unmixedDecomposition} 
requires no hypotheses on the base field, and holds even over
finite fields. Passing to an algebraic closure (actually, only a
finite extension is necessary) this result combines with our
computations of minimal primary components to give a clean, explicit, but
in general redundant, primary decomposition of any cellular binomial ideal.

\subsection*{Outline}
In Section~\ref{sec:preliminaries}, we review some known results about
binomial ideals, introduce our main tools, and precisely state our
main results. Section~\ref{sec:witnesses} is the technical core of this
paper, and contains results that are used to compute the minimal
primary components and hull of a cellular binomial ideal in
Section~\ref{sec:hull}, and the unmixed decomposition in Section~\ref{sec:unmixed}.
Section~\ref{sec:implications} explores the potential
consequences of the developments in this
article.

\subsection*{Acknowledgements}
We thank Thomas Kahle, Ezra Miller and Christopher O'Neill
for fruitful discussions; their helpful comments
improved a previous version of this article. We are also very grateful to the
anonymous referee for their thoughtful and insightful suggestions,
especially for Corollary~\ref{coro:associatedPrimes}
and for
the simplifications of the proofs of Corollary~\ref{coro:zeroMemb},
Proposition~\ref{propo:monomialsInI+Memb(I)} and
Theorem~\ref{thm:cellularPrimDecV1}.

\section{Preliminaries}
\label{sec:preliminaries}

Let $\kk$ be a field; throughout this article, $\bar{\kk}$ denotes
an algebraic closure of $\kk$.
A \emph{binomial} in $\kk[x]=\kk[x_1,\dots,x_n]$ is a polynomial that
has at most two terms. A \emph{binomial ideal} is an ideal generated
by binomials.

We use the convention that $0 \in \NN$.
If $\Delta \subseteq \{1,\dots,n\}$, we denote $\NN^\Delta = \{u \in
\NN^n \mid u_i = 0 \text{ for } i \notin \Delta\}$, and define
$\ZZ^\Delta$ analogously. We also use $\bDelta$ to denote the
complement $\{1,\dots,n\} \minus \Delta$. We utilize the usual
notation for monoid algebras: for instance, $\kk[\NN^\Delta] = \kk[x_i \mid i
\in \Delta]$.

A \emph{partial character on $\ZZ^n$} is a group homomorphism
$\rho:L_{\rho} \to \kk^*$, where $L_\rho$ is a subgroup of $\ZZ^n$ and
$\kk^*$ is the multiplicative group of the field $\kk$. We usually
specify partial characters on $\ZZ^n$ by giving the pair
$(L_\rho,\rho)$, unless the lattice $L_\rho$ is understood in context.

Given a partial character $(L_\rho,\rho)$ on $\ZZ^n$, we define the
corresponding \emph{lattice ideal} via
\[
I(\rho) \defeq \< x^u - \rho(u-v) x^v \mid  u-v \in L_\rho\> \subseteq \kk[x].
\]
Our notation here is not the same as in~\cite{ES}. In that article,
$I(\rho)$ indicates a lattice ideal in a Laurent polynomial ring,
while $I_+(\rho)$ is used for lattice ideals in $\kk[x]$. Since we do
not use Laurent polynomials in this work, we drop the subscript to
simplify the notation.

If the field $\kk$ is algebraically closed, by \cite[Theorem~2.1.c]{ES}
the lattice ideal arising
from a partial character $(L_\rho,\rho)$ on 
$\ZZ^n$ is prime if and only if the lattice $L_\rho \subseteq \ZZ^n$ is
\emph{saturated}, meaning that its \emph{saturation} 
\[
\Sat(L_\rho) \defeq (\QQ \otimes_{\ZZ} L_\rho) \cap \ZZ^n,
\]
is equal to $L_\rho$. If $L_\rho$ is
a saturated lattice, then $(L_\rho,\rho)$ is 
a \emph{saturated partial character} on $\ZZ^n$. If $(L_\rho,\rho)$ is
a partial character on $\ZZ^n$, any partial character $(\Sat(L_\rho),
\chi)$ such that the restriction of $\chi$ to $L_\rho$ is $\rho$, is
called a \emph{saturation} of $(L_\rho,\rho)$.

While lattice ideals arising from saturated partial
characters are prime, more is true;~\cite[Corollary~2.6]{ES} shows that 
if $\kk$ is algebraically closed, a
binomial ideal $P \subseteq \kk[x]$ is prime if and only if there exists a subset
$\Delta \subset \{1,\dots,n\}$ and a saturated partial character
$(L_\chi,\chi)$ on $\ZZ^\Delta$, such that
\[
P = \kk[x] \cdot I(\chi) + \<x_i \mid i \in \bDelta\>,
\]
where as before, $\bDelta = \{1,\dots,n\}\minus \Delta$.
Note that in this case, $I(\chi)$ is an ideal in $\kk[\NN^\Delta]$.

We now describe the primary decomposition of a lattice ideal.
Let $(L_\rho,\rho)$ be a partial character
on $\ZZ^n$, and let $p$ be a prime number.
We define $\Sat_{p}(L_{\rho})$ and $\Sat'_{p}(L_{\rho})$ to be the largest
sublattices of $\Sat(L_{\rho})$ containing $L_{\rho}$ such that
$|\Sat_p(L_{\rho})/L_{\rho} | = p^k$ for some $k \in \ZZ$, and
$| \Sat'_p(L_{\rho})/L_{\rho} | = g$ where $(p, g) = 1$. We also adopt the
convention that $\Sat_0(L_\rho) \defeq L_\rho$ and $\Sat'_0(L_\rho)
\defeq \Sat(L_\rho)$. 

\begin{theorem}[{\cite[Corollary~2.5]{ES}}]
\label{thm:latticePrimaryDecomposition}
Let $\kk$ be an algebraically closed field of characteristic $p \geq
0$, and let $(L_\rho,\rho)$ be a partial character on $\ZZ^n$.
There are $g = | \Sat'_p(L_\rho)/L_\rho|$ distinct partial characters
$\rho_1,\dots,\rho_g$ that extend $(L_{\rho},\rho)$ to $\Sat'_{p}(L_{\rho})$. For each
$i=1,\dots,g$, there exists a unique partial character $\chi_{i}$
that extends $\rho_{i}$ to  $\Sat(L_{\rho})$. (If $p=0$, $\chi_i =
\rho_i$ for all $i=1,\dots,g$.)
The associated primes of the lattice ideal $I(\rho)$ are
$I(\chi_1),\dots,I(\chi_g)$, they are all minimal, and have the same
codimension $\rank(L_\rho)$. For 
each $i=1,\dots,g$, the ideal $I(\rho_i)$ is $I(\chi_i)$-primary, and 
\[
I(\rho) = \bigcap_{i=1}^g I(\rho_i)
\]
is the minimal primary decomposition of $I(\rho)$. 
\qed
\end{theorem}

The above result implies that, over an
algebraically closed field of characteristic zero, a primary lattice
ideal is prime. This is not the case if $\chr(\kk)>0$; for instance
$\<x_1^2-x_2^2\> \subseteq \kk[x_1,x_2]$ is primary if $\chr(\kk)=2$,
as in this case $x_1^2-x_2^2=(x_1-x_2)^2$. However, this ideal is not
radical, and therefore not prime.

We now introduce the main objects of study in this article.

\begin{definition}
\label{def:cellular} 
An ideal $I\subseteq \kk[x]$ is \emph{cellular} if every
variable $x_i$ is either a nonzerodivisor modulo $I$ or is nilpotent
modulo $I$. The nonzerodivisor variables modulo a cellular binomial
ideal $I$ are called the
\emph{cellular variables of $I$}. If $I$ is a cellular binomial ideal,
and $\Delta \subseteq \{1,\dots,n\}$ indexes the cellular variables of
$I$, then $I$ is called \emph{$\Delta$-cellular}.
\endrk
\end{definition}

If $(L_\rho,\rho)$ is a partial character on $\ZZ^n$, then the lattice
ideal $I(\rho)$ is $\{1,\dots,n\}$-cellular. The following (well
known) result gives a kind of converse to this assertion.

\begin{lemma}
\label{lemma:latticeIdealOfCellular}
Let $I$ be a $\Delta$-cellular binomial ideal in $\kk[x]$. There
exists a partial character $(L_\rho,\rho)$ on $\ZZ^{\Delta}$ such that 
\[
I \cap \kk[\NN^{\Delta}] = I(\rho). 
\]
\end{lemma}

For the proof of Lemma~\ref{lemma:latticeIdealOfCellular}, we need to
understand the elimination ideals of a binomial ideal.

\begin{lemma}[{\cite[Corollary~1.3]{ES}}]
\label{lemma:elimination}
Let $I\subseteq \kk[x]$ be a binomial ideal, and let $\Delta \subseteq
\{1,\dots,n\}$. Then the elimination ideal $I \cap \kk[\NN^\Delta]$ is binomial.
\qed
\end{lemma}

\begin{proof}[Proof of Lemma~\ref{lemma:latticeIdealOfCellular}]
The ideal $I \cap \kk[\NN^{\Delta}] \subseteq
\kk[\NN^{\Delta}]$, which is a binomial ideal
by Lemma~\ref{lemma:elimination}, contains no monomials. Therefore, 
by~\cite[Corollary~2.5]{ES}, there exists a partial character
$(L_\rho,\rho)$ on $\ZZ^{\Delta}$ such that $((I \cap
\kk[\NN^{\Delta}]) : (\prod_{i \in \Delta} x_i)^{\infty} ) =
I(\rho)$. But since $I$ is $\Delta$-cellular, $((I \cap
\kk[\NN^{\Delta}]) : (\prod_{i \in \Delta} x_i)^{\infty} ) = I \cap
\kk[\NN^{\Delta}]$, and the result follows.
\end{proof}

The following result is crucially useful throughout.

\begin{lemma}[{\cite[Corollary~1.7.a]{ES}}]
\label{lemma:idealQuotient}
Let $I$ be a binomial ideal in $\kk[x]$ and let $m\in \kk[x]$ be a
monomial. Then the ideal quotient $(I:m)$ and the saturation $(I:m^\infty)$ are binomial ideals.
\qed
\end{lemma}

Following~\cite{KM2}, we introduce combinatorial objects attached to
a cellular binomial ideal.
Let $I$ be a $\Delta$-cellular binomial ideal in $\kk[x]$. 
A lattice $L \subseteq \ZZ^{\Delta}$ is \emph{associated}
to $I$ if there exists a \emph{witness monomial} $m \in
\kk[\NN^{\bDelta}]$ such that $(I : m ) \cap \kk[\NN^{\Delta}] =
I(\rho)$ for some partial character $\rho : L \rightarrow \kk^{*}$ on $\ZZ^{\Delta}$.
What follows is a characterization of the associated lattices to
a cellular binomial ideal.

\begin{theorem}
\label{thm:cellularWitnesses}
Let $I$ be a $\Delta$-cellular binomial ideal in
$\kk[x]$, where $\kk$ is algebraically closed. 
Assume that
$(L_\chi,\chi)$ is a saturated partial character on $\ZZ^\Delta$ such
that $\kk[x] \cdot I(\chi) + \< x_i \mid i \in \bDelta\>$ is an associated
prime of $I$. Then
there exist a monomial $m \in \kk[\NN^{\bDelta}]$ and a partial character
$(L_\tau,\tau)$ on $\ZZ^{\Delta}$ such that 
$(L_\chi,\chi)$ is a saturation of $(L_\tau,\tau)$, and  
\begin{equation}
\label{eqn:associatedLattice}
 (I : m) \cap \kk[\NN^{\Delta}] = I(\tau) .
\end{equation}
The converse of this statement also holds. If $m \in
\kk[\NN^\bDelta]$, and a partial character $(L_\tau,\tau)$ on
$\ZZ^\Delta$ is defined by~\eqref{eqn:associatedLattice}, then for any
associated prime $P$ of $I(\tau)$, $\kk[x] \cdot P + \<x_i \mid i \in
\bDelta\>$ is an associated prime of $I$. 
\end{theorem}

\begin{proof}
The first half of this statement
is~\cite[Theorem~8.1]{ES}. Theorem~15.11 in~\cite{KM} generalizes this
result, and also provides a converse.
\end{proof}

The following key object attached to a cellular binomial ideal is used
throughout this work.

\begin{definition} 
\label{def:embedded+Memb}
Let $I$ be a $\Delta$-cellular
binomial ideal in $\kk[x]$, and let $(L_\rho,\rho)$ be the partial character on
$\ZZ^\Delta$ that satisfies $I\cap \kk[\NN^{\Delta}] = I(\rho)$. A
lattice $L$ associated to $I$ is said to be
\emph{embedded} if $\rank(L)>\rank(L_{\rho})$, equivalently, if
$L$ properly contains $L_{\rho}$ and
$\Sat(L_{\rho}) \neq \Sat(L)$.
We define 
\[
\Memb(I) \defeq \< \text{witness monomials of embedded associated
  lattices to } I \>.
\]
If $I$ has no embedded
associated lattices, $\Memb(I)=0$.
We remark that by construction, $1 \notin \Memb(I)$, and $\Memb(I)$ is
generated by monomials in $\kk[\NN^\bDelta]$.
\endrk
\end{definition}

\begin{remark}
\label{rmk:relationshipToKM2}
The above definition is an extension of a construction of
Kahle~\cite{KM2}, that has its roots in the work of Kahle and
Miller~\cite{KM}. We note that in~\cite{KM2}, the definition of
embedded associated lattice requires only proper containment of the
lattices involved;
however, in the case of interest in~\cite{KM2}, when $I$
is $\Delta$-cellular with a unique minimal prime (over the algebraic
closure of $\kk$) and $\chr(\kk)=0$, proper containment of those
lattices implies that their ranks are not the same.

The notion of \emph{associated mesoprime} from~\cite{KM} is an
evolved version of the concept of associated lattices that is used in
this article (see~\cite[Remark~12.8]{KM}). In the case of cellular
binomial ideals, every associated mesoprime is supported on an
associated lattice. 
The associated mesoprimes of a (cellular) binomial ideal $I$ arise
from \emph{essential $I$-witnesses} (see~\cite[Definition~12.1]{KM}), which
are witness monomials as defined above, satisfying additional
maximality conditions.
\endrk
\end{remark}

\begin{example} 
\label{example:Memb} 
  The binomial ideal $I = \langle x_1^{3}x_3-x_1^{3},
  x_1^{4}, x_1^{2}x_2x_4-x_1^{2}x_2, x_2^{2}, x_4^{3}-1 \rangle \subseteq
  \kk[x_1,x_2,x_3,x_4]$ is $\{3,4\}$-cellular. In
  this case $\Memb(I) = \langle x_1^{3} \rangle$. Observe that $x_1^{2}x_2$
  is not in $\Memb(I)$: even though $(I : x_1^{2}x_2) \cap \kk[x_3,x_4]$ contains
  the binomial $x_4 - 1$, the fact that $x_4^{3} - 1 \in I$ implies that
  the lattices corresponding to the lattice ideals obtained by
  intersecting $I$ and $(I:x_1^2x_2)$ with $\kk[x_3,x_4]$, have the same
  saturation. 
\endrk
\end{example}

We are now ready to state our main results.

\begin{theorem}
\label{thm:hullDecomposition}
Let $I$ be a $\Delta$-cellular binomial ideal in $\kk[x]$. Denote by $\hull(I)$
the intersection of the minimal primary components of $I$.
Then $\hull(I) = I +
\Memb(I)$. (In particular, $\hull(I)$ is a binomial ideal, cf.~\cite[Corollary~6.5]{ES}.) Assume now
that $\kk$ is algebraically closed. 
If $(L_\rho,\rho)$ is a partial character on 
$\ZZ^\Delta$ such that $I \cap \kk[\NN^\Delta] = I(\rho)$, and $g$ and
$\rho_1,\dots,\rho_g$ are as in
Theorem~\ref{thm:latticePrimaryDecomposition}, so that 
$I(\rho) = \cap_{\ell=1}^g I(\rho_\ell)$ is a minimal primary
decomposition, then
\[
\hull(I) = I + \Memb(I) = \bigcap_{\ell=1}^g
\bigg[ \bigg((I+I(\rho_\ell)): (\prod_{i\in \Delta} x_i)^\infty\bigg) +
\Memb(I) \bigg]
\] 
is a minimal primary decomposition of $\hull(I)$.
\end{theorem}

\begin{proof}
This is a combination of
Theorems~\ref{thm:cellularPrimDecV2} and~\ref{thm:cellularHull}.
\end{proof}

We have abused notation above and used
$I(\rho_\ell)$ to also denote the extension $\kk[x] \cdot
I(\rho_\ell)$; we continue this practice where it causes no confusion.

Recall that an ideal $I$ is \emph{unmixed} if the codimensions of all of its
associated primes are equal. By Corollary~\ref{coro:unmixedForCellular},
a cellular binomial ideal is 
unmixed if and only if all of its associated primes are
minimal. 
The following result provides an unmixed decomposition for a cellular
binomial ideal. The characteristic zero version
of~\eqref{eqn:unmixedDecomposition} was proved by Eisenbud and
Sturmfels (\cite[Corollary~8.2]{ES}), who conjectured it also holds in
positive characteristic. We prove this conjecture in
Section~\ref{sec:unmixed}, and combine it with 
Theorem~\ref{thm:hullDecomposition} to give an explicit formula for a primary
decomposition of a cellular binomial ideal.

\begin{theorem}
\label{thm:unmixedDecomposition}
Let $I$ be a $\Delta$-cellular binomial ideal in $\kk[x]$. If $m \in
\Memb(I)$ or $m=1$,
let $(L_{\rho(m)},\rho(m))$ be the partial character on $\ZZ^\Delta$
such that $(I:m) \cap \kk[\NN^\Delta] = I(\rho(m))$. Then
\begin{align}
I 
& =  \bigcap_{m \in \Memb(I)\cup\{1\}} \hull\bigg( (I+I(\rho(m))) :
(\prod_{i\in\Delta} x_i)^\infty \bigg)  
\label{eqn:unmixedDecomposition} \\
\label{eqn:unmixedDecompositionExplicit}
& = 
  \bigcap_{m \in \Memb(I)\cup\{1\}} \bigg[ \bigg( (I + I(\rho(m))) : (\prod_{i\in
  \Delta} x_i)^\infty\bigg) + \Memb\bigg( (I + I(\rho(m))) : (\prod_{i\in
  \Delta} x_i)^\infty \bigg) \bigg] 
\end{align}
Assume now that $\kk$ is algebraically closed.
If $g(m)$,
$\rho(m)_1,\dots,\rho(m)_{g(m)}$ are as in
Theorem~\ref{thm:latticePrimaryDecomposition}, so that 
$I(\rho(m)) = \cap_{\ell=1}^{g(m)} I(\rho(m)_\ell)$ is a minimal primary
decomposition, then
\[
I = \bigcap_{m \in \Memb(I) \cup\{1\}} \bigcap_{\ell=1}^{g(m)} 
\bigg[ \bigg((I+I(\rho(m)_\ell)): (\prod_{i\in \Delta} x_i)^\infty\bigg) +
\Memb\bigg((I+I(\rho(m))):(\prod_{i\in\Delta} x_i)^\infty \bigg) \bigg]
\]
is a (not necessarily minimal) primary decomposition of $I$.
\end{theorem}

One of the reasons that binomial primary decomposition is more transparent in
characteristic zero is that, in that case, the $P$-primary component
of a binomial ideal $I$ contains the saturated lattice ideal $P \cap
\kk[x_j \mid x_j \notin P]$, while 
over positive characteristic it is difficult to determine what
lattice ideal takes the place of $P \cap \kk[x_j \mid x_j \notin
P]$. It is this question that Theorem~\ref{thm:unmixedDecomposition}
seeks to answer for cellular binomial ideals.

In~\cite[Section~4]{OS}, Ojeda and Piedra provide an effectively
computable unmixed decomposition for a cellular binomial ideal (over
any field), which has been implemented by Kahle~\cite{KM2}. This
decomposition proceeds by computing ideal quotients with respect to
well chosen binomials. In this case, the ideal quotient is still
binomial, and can be found by adding an explicit monomial ideal. 
The unmixed decomposition procedure~\cite[Algorithm 4]{OS} is a
recursive method in the same spirit as the cellular decomposition of a
binomial ideal, and likewise shares some of its drawbacks: many choices need
to be made, and the output is therefore not canonical.

We point out that the unmixed decomposition of cellular binomial
ideals~\eqref{eqn:unmixedDecompositionExplicit} is 
usually coarser than (the cellular case of) the mesoprimary
decompositions of Kahle and Miller, as is shown
in~\cite[Example~15.13]{KM}. Nevertheless,~\eqref{eqn:unmixedDecompositionExplicit}
does share some of the features of mesoprimary decomposition: it is
canonical, in the sense that no choices need to be 
made in order to perform it; it requires no assumptions on the field
$\kk$; and it contains enough information to easily obtain a primary
decomposition from it, when working over an algebraic closure
$\bar{\kk}$ of $\kk$.

\section{Associated lattices and witness monomials}
\label{sec:witnesses}

In this section, we study the ideals $\Memb(I)$ and  $I+\Memb(I)$ (see
Definition~\ref{def:embedded+Memb}) for a cellular binomial ideal $I$.

As a first step, we use Theorem~\ref{thm:latticePrimaryDecomposition} and
Lemma~\ref{lemma:latticeIdealOfCellular} to characterize 
the minimal associated primes of a cellular binomial ideal.

\begin{lemma}
\label{lemma:minimalPrimesOfCellular}
Let $I$ be a $\Delta$-cellular binomial ideal in $\kk[x]$. Let
$(L_\rho,\rho)$ be a partial character on $\ZZ^\Delta$ such that
$I\cap \kk[\NN^\Delta]=I(\rho)$. The minimal associated primes of $I$
are of the form $\kk[x]\cdot P + \<x_i \mid i \in \bDelta\>$, where
$P \subset \kk[\NN^\Delta]$ is a minimal associated prime of $I(\rho)$.
If $\kk$ is algebraically closed, the
minimal associated primes of $I$ are
\[
\kk[x] \cdot I(\chi_j) + \<x_i \mid i\in \bDelta \> \quad \text{for }
j=1,\dots, g,
\]
where $g$ and $\chi_1,\dots,\chi_g$ are given by
Theorem~\ref{thm:latticePrimaryDecomposition} applied to the lattice
ideal $I(\rho)$.
\end{lemma}

\begin{proof}
Let $Q \subset \kk[x]$ be a prime ideal containing $I$.
Since $I$ is $\Delta$-cellular, $Q$ must
contain the variables $x_i$ for $i \in \bDelta$. 
Also, the prime ideal $Q\cap \kk[\NN^\Delta]$ contains $I(\rho)$, and therefore
contains some minimal prime $P$ of
$I(\rho)$, so that
$Q \supseteq \kk[x] \cdot P + \<x_i \mid i
\in \bDelta\>$, which gives the desired form for the minimal primes
over $I$.
\end{proof}

We turn our attention now to the monomial ideal $\Memb(I)$ for a
$\Delta$-cellular binomial ideal $I$.
The following result gives a useful criterion to determine whether a
monomial belongs to $\Memb(I)$.

\begin{lemma}
\label{lemma:monomialsInMemb}
Let $I$ be a $\Delta$-cellular binomial ideal in $\kk[x]$, and let
$(L_{\rho},\rho)$ be the partial character on $\ZZ^\Delta$ that
satisfies $I \cap \kk[\NN^\Delta] = I(\rho)$. 
A monomial $m \in \kk[\NN^\bDelta] \minus I$ belongs to $\Memb(I)$
if and only if there exists a binomial $x^u-\lambda x^v \in
\kk[\NN^\Delta]$ such that
\begin{enumerate}
\item $\lambda \neq 0$,
\item $u-v \notin \Sat(L_\rho)$, and
\item $m(x^u-\lambda x^v) \in I$.
\end{enumerate}
Moreover, in this case, $mx^u,mx^v \notin I$, and we may assume $\gcd(x^u,x^v)=1$.
\end{lemma}

\begin{proof}
If $m \in \Memb(I)$, there exists a partial character $(L_\tau,\tau)$
on $\ZZ^\Delta$ such that $(I:m) \cap \kk[\NN^\Delta] = I(\tau)$, and
$\rank(L_\tau) > \rank(L_\rho)$. This rank condition implies that we may 
choose $x^u-\lambda 
x^v \in I(\tau) \subseteq \kk[\NN^\Delta]$ such that 
$u-v \notin \Sat(L_\rho)$. 
%
Since $I(\tau)$ contains no monomials, we see that $\lambda \neq 0$.
Moreover $m x^u, mx^v \notin I$ by cellularity, because $m \notin I$ and $x^u, x^v \in
\kk[\NN^{\Delta}]$. 
Similarly, we may assume that $\gcd(x^u,x^v)=1$.

For the converse, let $m \notin I$ be a monomial in
$\kk[\NN^{\bDelta}]$, and suppose there is a binomial $x^u-\lambda x^v 
\in \kk[\NN^\Delta]$ satisfying the three conditions required
above.
By the third condition, $x^u-\lambda x^v \in (I:m)\cap
\kk[\NN^\Delta]$. The binomial ideal $(I:m)$ 
is $\Delta$-cellular, so that $(I:m)\cap \kk[\NN^\Delta]
= I(\tau)$ for some partial character $(L_\tau,\tau)$ on $\ZZ^\Delta$.  
Since $(I:m) \supseteq I$, we have $I(\tau) \supseteq I(\rho)$, and
therefore $L_\tau \supseteq  L_\rho$. Moreover, $u-v \in L_\tau \minus
\Sat(L_\rho)$, which implies that $L_\tau$ is an embedded
associated lattice to $I$ and $m\in\Memb(I)$.
\end{proof}

If $I$ is a cellular binomial ideal over an algebraically closed
field, Theorem~\ref{thm:cellularWitnesses} gives a correspondence
between the associated 
primes of $I$  and the associated
lattices of $I$. The following result allows us to remove the
base field assumption on this correspondence (see Corollary~\ref{coro:associatedPrimes}).

\begin{theorem}
\label{thm:associatedPrimesFieldExtension}
Let $J \subset \kk[x]$ be a proper ideal, and let $\bar{J}$ be the
extension of $J$ to $\bar{\kk}[x]$, where $\bar{\kk}$ is an algebraic
closure of $\kk$. Let $P \subset \kk[x]$ be a prime ideal
containing $J$, and let $\bar{P} \subset \bar{\kk}[x]$ be a prime
ideal lying over $P$. (Given $P$, the existence of $\bar{P}$ follows
since $\bar{\kk}[x]$ is integral over $\kk[x]$.) The following hold.
\begin{enumerate}
\item The ideals $P$ and $\bar{P}$ have the same codimension.
\item $P$ is an associated prime of $J$ if and only if $\bar{P}$ is
  an associated prime of $\bar{J}$.
\item $P$ is a minimal prime of $J$ if and only if $\bar{P}$ is a
  minimal prime of $\bar{J}$.
\item $J$ is unmixed if and only if $\bar{J}$ is unmixed.
\end{enumerate} 
\end{theorem}

\begin{proof}
The first item holds since $\bar{\kk}[x]$ is integral over $\kk[x]$,
and $\kk[x]$ is a normal domain.
The second and third items are proved in~\cite[Theorem~2.5]{Sh} (see
also~\cite{S}). The last item is a combination of the first two.
\end{proof}

\begin{corollary}
\label{coro:MembAlgebraicClosure}
Let $I$ be a proper $\Delta$-cellular binomial ideal in $\kk[x]$ and let
$\bar{I}$ be the extension of $I$ to $\bar{\kk}[x]$, where $\bar{\kk}$
is an algebraic closure of $\kk$. Then $\bar{I}$ is a
proper $\Delta$-cellular binomial ideal, and the extension of
$\Memb(I)$ to $\bar{\kk}[x]$ is $\Memb(\bar{I})$. 
\end{corollary}

\begin{proof}
To see that $\bar{I} \subset \bar{\kk}[x]$ is proper, observe that
$\bar{I}\cap \kk[x]=I$. This uses the fact that a system of linear
equations defined over $\kk$ that has a solution over an extension
field of $\kk$, must have a solution over $\kk$.
Similarly, if $m$ is a monomial, $(\bar{I}:m) \cap \kk[x] = (I:m)$. It
follows that $\bar{I}$ is $\Delta$-cellular. 

We wish to show that $\Memb(I)$ and $\Memb(\bar{I})$ contain the same
monomials. 

Let $m\neq 0$ be a minimal monomial generator of $\Memb(\bar{I})$. Let
$(L_{\bar{\tau}},\bar{\tau})$ be a character on $\ZZ^\Delta$ with
$\bar{\tau}:\ZZ^\Delta \to \bar{\kk}^*$ such that  
$(\bar{I}:m) \cap \bar{\kk}[\NN^\Delta] = I(\bar{\tau})$, and let
$(L_\tau,\tau)$ with $\tau : \ZZ^\Delta \to \kk^*$ such that $(I:m)
\cap \kk[\NN^\Delta] = I(\tau)$. Then $I(\tau) = I(\bar{\tau})\cap \kk[\NN^\Delta]$, so
that $L_\tau \subset L_{\bar{\tau}}$ and $\bar{\tau}$ restricts to
$\tau$ on $L_\tau$. By
Theorem~\ref{thm:associatedPrimesFieldExtension}, $I(\tau)$ and
$I(\bar{\tau})$ have the same codimension, and therefore
by~\cite[Theorem~2.1]{ES}, the lattices $L_\tau$ and $L_{\bar{\tau}}$
have the same rank, and so do $\Sat(L_\tau)$ and
$\Sat(L_{\bar{\tau}})$.

If we define partial characters $(L_\rho,\rho)$ and
$(L_{\bar{\rho}},\bar{\rho})$  on $\ZZ^\Delta$ via $I\cap \kk[\NN^\Delta] = I(\rho)$ and
$\bar{I} \cap \bar{\kk}[\NN^\Delta] = I(\bar{\rho})$, then as before, $L_\rho$ and
$L_{\bar{\rho}}$ have the same 
rank, and so do $\Sat(L_\rho)$ and $\Sat(L_{\bar{\rho}})$. 
Since $m \in \Memb(\bar{I})$ is a minimal generator, we have $\Sat(L_{\bar{\rho}}) \neq
\Sat(L_{\bar{\tau}})$, which implies that $\Sat(L_{\bar{\rho}})$ has
strictly lower rank than $\Sat(L_{\bar{\tau}})$. But then the same
holds for $\Sat(L_\rho)$ and $\Sat(L_\tau)$, and consequently
$\Sat(L_\rho) \neq \Sat(L_\tau)$. We conclude that $m\in \Memb(I)$.

A similar argument shows the reverse inclusion.
%
%
\end{proof}

\begin{corollary}
\label{coro:associatedPrimes}
Let $I$ be a $\Delta$-cellular binomial ideal in $\kk[x]$. 
The
associated primes of $I$ are the ideals $\kk[x] \cdot P + \< x_i \mid i \in
\bDelta\>$, where $P \subset \kk[\NN^\Delta]$ runs over the associated
primes of lattice ideals of the form $(I:m)\cap \kk[\NN^\Delta]$, for
monomials $m
\in \kk[\NN^\bDelta]$.
\end{corollary}

\begin{proof}
By Corollary~\ref{coro:MembAlgebraicClosure},
Theorem~\ref{thm:cellularWitnesses} can be applied to the extension
$\bar{I}$. Now the result follows from Theorem~\ref{thm:associatedPrimesFieldExtension}.
\end{proof}

The following two results give us characterizations of unmixedness for
cellular binomial ideals. The first one has already appeared
as~\cite[Proposition~2.4]{OS}.

\begin{corollary}
\label{coro:unmixedForCellular}
Let $I$ be a $\Delta$-cellular binomial ideal in $\kk[x]$. Then $I$ is
unmixed if and only if it has no embedded associated primes.
\end{corollary}

\begin{proof}
By Theorem~\ref{thm:associatedPrimesFieldExtension} and 
Corollary~\ref{coro:MembAlgebraicClosure}, we
may assume that $\kk$ is algebraically closed. Now the result follows
from Lemma~\ref{lemma:minimalPrimesOfCellular}, since the minimal
primes of $I$ have the same codimension.
\end{proof}

\begin{corollary} 
\label{coro:zeroMemb} 
Let $I$ be a $\Delta$-cellular binomial ideal in $\kk[x]$. Then $I$ is
unmixed if and only if $\Memb(I)=0$.
\end{corollary}

\begin{proof}
By Lemma~\ref{lemma:minimalPrimesOfCellular} and
Corollary~\ref{coro:associatedPrimes}, $I$ has embedded primes if and
only if it has embedded associated lattices. Now use Corollary~\ref{coro:unmixedForCellular}.
\end{proof}

We state an immediate consequence of Corollary~\ref{coro:zeroMemb}.

\begin{corollary} 
\label{coro:primaryCellular}
Let $I$ be a cellular binomial ideal in $\kk[x]$. If $I$ has only one minimal
associated prime and $\Memb(I) = 0$, then $I$ is primary. 
\qed
\end{corollary}

We now turn our attention to the ideal $I+\Memb(I)$, and determine all
of its monomials.

\begin{proposition}
\label{propo:monomialsInI+Memb(I)}
Let $I$ be a cellular binomial ideal.
Any monomial in $I+\Memb(I)$ belongs to either $I$ or $\Memb(I)$.
\end{proposition}

In order to prove this result, we need the following auxiliary lemma.

\begin{lemma}[{\cite[Corollary~1.6.(b)]{ES}}]
\label{lemma:binomial+monomial}
Let $I$ be a binomial ideal in $\kk[x]$, and let $m,m_1,\dots,m_s \in
\kk[x]$ be monomials. If $m \in I+\<m_1,\dots,m_s\>$, then $m \in
I+\<m_i\>$ for some $i$.
\qed
\end{lemma}

\begin{proof}[Proof of Proposition~\ref{propo:monomialsInI+Memb(I)}]
Let $I$ be a $\Delta$-cellular binomial ideal.
Let $x^{\mu} \in I + \Memb(I)$, and assume that $x^{\mu} \notin I$. 
By Lemma~\ref{lemma:binomial+monomial}, there exists a generator $m$ of
$\Memb(I)$ and a nonzero $\lambda \in \kk$ such that $x^{\mu} -
\lambda m \in I$. We may assume that $m \in \kk[\NN^\bDelta]$.

Since $x^\mu \notin I$, $m \notin I$.
Then, by Lemma~\ref{lemma:monomialsInMemb}, $m \in (\Memb(I) \cap
\kk[\NN^\bDelta]) \minus I$ implies that
there exists a binomial $x^u-\kappa x^{u'} \in \kk[\NN^\Delta]$ such
that $\kappa \neq 0$, $u-u' \notin \Sat(L_\rho)$, and $m(x^u-\kappa
x^{u'} ) \in I$.
Consider the product
\[
(x^\mu  - \lambda m)(x^u-\kappa x^{u'}) = x^\mu  (x^u-\kappa
x^{u'}) -\lambda m(x^u-\kappa x^{u'}).
\]
Since $x^\mu - \lambda x^v$ and $m (x^u-\kappa x^{u'})$ belong to $I$,
we have $x^\mu (x^u-\kappa x^{u'}) \in I$. 
Applying Lemma~\ref{lemma:monomialsInMemb} again, we see that $x^\mu \in
\Memb(I)$, which concludes the proof.
\end{proof}

The following is one of our core technical results.

\begin{theorem} 
\label{thm:preliminaryToHull} 
Let $I$ be a $\Delta$-cellular binomial ideal in $\kk[x]$, and let $(L_\rho,\rho)$
be a partial character on $\ZZ^\Delta$ such that $I \cap
\kk[\NN^\Delta] = I(\rho)$. Then the binomial ideal $I + \Memb(I)$ is
$\Delta$-cellular, and we have 
\[
\Memb(I+\Memb(I)) = 0 \, ; \quad \text{and} \quad
(I + \Memb(I)) \cap \kk[\NN^\Delta] = I(\rho).
\]
Consequently, all associated primes of $I + \Memb(I)$ are minimal, 
and coincide with the minimal associated primes of $I$.
\end{theorem}

In the proof of Theorem~\ref{thm:preliminaryToHull}, the
following result will be useful.

\begin{proposition}[{\cite[Proposition~1.10]{ES}}]
\label{propo:binomial+monomial}
Let $I$ be a binomial ideal and $M$ a monomial ideal in
$\kk[x]$. If $f \in I + M$ and $f'$ is the sum of the
terms of $f$ that are not individually contained in $I + M$, then $f'
\in I$. \qed
\end{proposition}

\begin{proof}[Proof of Theorem~\ref{thm:preliminaryToHull}]
To see that $I+\Memb(I)$ is $\Delta$-cellular, it is enough to show
that, for $j\in \Delta$, $((I + \Memb(I)) : x_{j}) = I + \Memb(I)$. This
follows if we show that all monomials and binomials 
contained in $((I + \Memb(I)) : x_{j})$ also belong to $I+\Memb(I)$.

%

Let $x^{\mu}$ be a monomial in $((I + \Memb(I)) : x_{j})$.  Then
$x_jx^\mu \in I + \Memb(I)$. By
Proposition~\ref{propo:monomialsInI+Memb(I)}, $x_j x^\mu$ belongs to
$I$ or to $\Memb(I)$.
In the first case, $x^\mu \in
I$ since $j\in \Delta$.
In the second
case, $x^\mu \in \Memb(I)$, since $\Memb(I)$ is generated by monomials
in $\kk[\NN^\bDelta]$. Thus $x^\mu \in I+\Memb(I)$.

Now let $x^u-\lambda x^v \in ((I+\Memb(I)):x_j)$, so that $x_j(x^u-\lambda x^v) \in I+\Memb(I)$,
where $\lambda \neq
0$; we wish to show that 
$x^u-\lambda x^v \in I + \Memb(I)$. By the previous argument, we may
assume $x_jx^u, x_jx^v \notin I + \Memb(I)$. We apply
Proposition~\ref{propo:binomial+monomial} to see that $x_j(x^u-\lambda
x^v) \in I$, and use $j \in \Delta$ to conclude $x^u-\lambda x^v \in I
+ \Memb(I)$. This finishes the proof that $I+\Memb(I)$ is $\Delta$-cellular.

%

Let us verify that $(I+\Memb(I)) \cap \kk[\NN^\Delta] =
I(\rho)$. The inclusion $\supseteq$ is clear. Since $I+\Memb(I)$ is
$\Delta$-cellular, 
$(I+\Memb(I)) \cap \kk[\NN^\Delta]$ is a binomial ideal that contains
no monomials. Therefore, the reverse inclusion follows if we show that
any binomial $b \in (I+\Memb(I)) \cap 
\kk[\NN^{\Delta}]$ belongs also to $I(\rho)$. Since $b \in
\kk[\NN^\Delta]$, the monomials of $b$ do not individually belong to
$\Memb(I)$, and by Proposition~\ref{propo:binomial+monomial}, $b \in I
\cap \kk[\NN^\Delta] = I(\rho)$, as we wanted.

%

Next we show that $\Memb(I+\Memb(I)) = 0$. 

By contradiction, let $x^{\mu} \in \Memb(I + \Memb(I)) \cap
\kk[\NN^{\bDelta}]$. Then $x^\mu \notin I+\Memb(I)$. Use
Lemma~\ref{lemma:monomialsInMemb} to find a binomial $x^u-\lambda x^v
\in \kk[\NN^\Delta]$
such that $\lambda \neq 0$, $u-v \notin \Sat(L_\rho)$, and
$x^\mu(x^u-\lambda x^v) \in I + \Memb(I)$. (Here we use the fact that
$(I+\Memb(I))\cap\kk[\NN^\Delta] = I(\rho)$.)

We may assume that $x^\mu x^u, x^\mu x^v \notin I+\Memb(I)$, so by
Proposition~\ref{propo:binomial+monomial}, $x^\mu(x^u-\lambda x^v) \in
I$. This implies that $x^\mu \in \Memb(I)$
by Lemma~\ref{lemma:monomialsInMemb}, a contradiction.


In order to prove that all associated primes of $I + \Memb(I)$ are
minimal we use Corollaries~\ref{coro:zeroMemb}
and~\ref{coro:unmixedForCellular}. By
Lemma~\ref{lemma:minimalPrimesOfCellular}, the minimal
primes over $I$ coincide with the minimal primes over $I+\Memb(I)$.
\end{proof}

The following result is a consequence of the last assertion of
Theorem~\ref{thm:preliminaryToHull} (see also Corollary~\ref{coro:primaryCellular}).

\begin{corollary}
\label{coro:computationallyUseful}
Let $I$ be a cellular binomial ideal in $\kk[x]$ with a unique minimal prime $P$. Then
$I + \Memb(I)$ is $P$-primary. \qed
\end{corollary}

\section{The hull of a cellular binomial ideal}
\label{sec:hull}

In this section, we compute the minimal primary components of a
cellular binomial ideal. This information is then used to give an
explicit formula for the hull of such an ideal. While the description
of the primary components requires the field $\kk$ to be algebraically
closed, the hull computation needs no assumptions on the base field.

\begin{theorem} 
\label{thm:cellularPrimDecV1}
Let $I$ be a $\Delta$-cellular binomial ideal in $\kk[x]$, where $\kk$
is algebraically closed and $\chr(\kk)= p \geq 0$. Let $(L_\rho,\rho)$ be the
partial character on $\ZZ^\Delta$ such that $I\cap \kk[\NN^\Delta] =
I(\rho)$. Consider $g$, $\rho_1,\dots,\rho_g$ and $\chi_1,\dots,\chi_g$ as
in Theorem~\ref{thm:latticePrimaryDecomposition}, so that
$I(\rho_\ell)$ is the $I(\chi_\ell)$-primary component of $I(\rho)$
for $\ell = 1,\dots,g$. By Lemma~\ref{lemma:minimalPrimesOfCellular},
the minimal associated primes of $I$ are 
\[
P_\ell = \kk[x] \cdot I(\chi_\ell) + \< x_i \mid i \in \bDelta\> \quad
\text{for} \; \ell=1,\dots,g. 
\]
For $\ell=1,\dots, g$, the $P_\ell$-primary component of $I$ is
\begin{equation}
\label{eqn:cellularPrimDecV1}
\bigg( (I + I(\rho_\ell)) : (\prod_{i \in \Delta} x_i)^\infty \bigg) + \Memb \bigg( (I + I(\rho_\ell)) : (\prod_{i \in \Delta} x_i)^\infty \bigg).
\end{equation}
\end{theorem}

\begin{example}
\label{ex:needSaturation}
The ideals $I+I(\rho_\ell)$ in Theorem~\ref{thm:cellularPrimDecV1} need not be
cellular; this is the reason we saturate out the variables indexed by
$\Delta$ in~\eqref{eqn:cellularPrimDecV1}. We illustrate this
phenomenon in the following example. Let $I = \< x_1^2-x_2^2,
x_3(x_1-x_2),x_3^3\>$; this is a $\{1,2\}$-cellular
binomial ideal, with $\Memb(I)=0$, and  $I \cap \kk[x_1,x_2] =
\<x_1^2-x_2^2\>$. If $\chr(\kk)=2$, this lattice ideal is primary, as
$x_1^2-x_2^2 = (x_1-x_2)^2$, and therefore $I$ is also primary by
Lemma~\ref{lemma:primaryBinomialIdeal}.
If $\chr(\kk)\neq 2$, then $I \cap \kk[x_1,x_2] = \<x_1-x_2\> \cap
\<x_1+x_2\>$, and we consider:
\begin{align*}
J_1 & = ((I + \<x_1 - x_2 \>):(x_1x_2)^\infty) = \<x_1-x_2,x_3^3
\>; \\
J_2 & = ((I + \<x_1 + x_2 \>):(x_1x_2)^\infty) = \<x_1+x_2,x_3\>.
\end{align*}
The ideal $I+\<x_1+x_2\>$ contains the monomial $x_2x_3$ without
containing $x_3$, and so the variable $x_2$ is neither nilpotent nor a
nonzerodivisor modulo this ideal, whence $I+\<x_1+x_2\>$ is not
cellular.

In this case, $\Memb(J_1)=\Memb(J_2) = 0$, and (since $I$ has no
embedded primes by Corollaries~\ref{coro:unmixedForCellular}
and~\ref{coro:zeroMemb}) we see that $I = J_1 \cap J_2$ 
is an irredundant primary decomposition.

This kind of saturation can also introduce binomials. Let $I = \<
x_1^2-x_2^2, x_2x_3-x_1x_4, x_1x_3-x_2x_4, x_3^2-x_4^2,
x_3x_4^2,x_4^3\>$, an unmixed
$\{1,2\}$-cellular ideal (note $x_3^3 \in I$), and assume $\chr(\kk) \neq 2$. In this case,
the primary components of $I$ are
\begin{align*}
J_1 & = ((I + \<x_1 - x_2 \>):(x_1x_2)^\infty) = \<x_1-x_2,x_3-x_4,x_4^3\>; 
\\
J_2 & = ((I + \<x_1 + x_2 \>):(x_1x_2)^\infty) = \<x_1+x_2,x_3+x_4,x_4^3\>.
\end{align*}
Note that $I + \<x_1 - x_2 \>$ contains $x_2(x_3-x_4)$ but not
$x_3-x_4$, and $I + \<x_1 + x_2 \>$ contains $x_2(x_3+x_4)$ but not $x_3+x_4$.
\endrk
\end{example}

\begin{proof}[{Proof of Theorem~\ref{thm:cellularPrimDecV1}}] 
Fix $1 \leq \ell \leq g$ and set $J_\ell = \big( (I + I(\rho_\ell)) : (\prod_{i \in \Delta}
x_i)^\infty \big)$. By construction, the variables indexed by $\Delta$
are nonzerodivisors modulo $J_\ell$. Also, since $I \subseteq J_\ell$,
the variables indexed by $\bDelta$ are nilpotent modulo
$J_\ell$. Thus, $J_\ell$ is $\Delta$-cellular.

Note that $I(\rho_\ell) \subseteq J_\ell \cap \kk[\NN^\Delta]
\subseteq I(\chi_\ell)$. By
Theorem~\ref{thm:latticePrimaryDecomposition}, this implies $J_\ell
\cap \kk[\NN^\Delta]$ is $I(\chi_\ell)$-primary. 
Now we apply 
Lemma~\ref{lemma:minimalPrimesOfCellular} to see that $J_\ell$ has a unique
minimal prime $P_\ell$, and conclude $J_\ell+\Memb(J_\ell)$ is $P_\ell$-primary
by Corollary~\ref{coro:computationallyUseful}.

%

We know that $I(\rho_\ell)$ is the primary component of $I(\rho)$
corresponding to its (minimal) associated prime $I(\chi_\ell)$
(see Theorem~\ref{thm:latticePrimaryDecomposition}). This implies that
  $I(\rho_\ell)$ is the kernel of the localization map
  $\kk[\NN^\Delta]/I(\rho) \to
  (\kk[\NN^\Delta]/I(\rho))_{I(\chi_\ell)}$ that inverts the elements
  in $\kk[\NN^\Delta]/I(\rho)$ outside the prime ideal $(\kk[\NN^\Delta]/I(\rho))
  \cdot I(\chi_\ell)$ . Therefore $I(\rho_\ell)$ is
  contained in the $P_\ell$-primary component $Q_\ell$ of $I$, and thus
   $I+
   I(\rho_\ell) \subseteq Q_\ell$, so that $J_\ell \subseteq (Q_\ell :
   (\prod_{i \in \Delta} x_i)^\infty ) = Q_\ell$.

We claim that $\Memb(J_\ell) \subseteq Q_\ell$. Note that this is
enough to conclude that $J_\ell + \Memb(J_\ell) = Q_\ell$, as we
wanted, since a $P_\ell$-primary ideal contained in the kernel of the
localization homomorphism $\kk[x]/I \to (\kk[x]/I)_{P_\ell}$ must
equal this kernel by~\cite[Corollary~10.21]{AM}. 


Recall that $J_\ell \cap \kk[\NN^\Delta]=I(\rho_\ell)$, whose
corresponding lattice has saturation $\Sat(L_\rho)$.
Let $m \in \kk[\bDelta] \minus J_\ell$ be a (monomial) generator of
$\Memb(J_\ell)$, and let
$x^u - \lambda x^v \in \kk[\NN^\Delta]$ be the binomial from 
Lemma~\ref{lemma:monomialsInMemb} applied to $m$ and $J_\ell$.
Since $m(x^u-\lambda x^v) \in J_\ell$, this binomial maps to zero
under $\kk[x]/I \to (\kk[x]/I)_{P_\ell}$. Now note that $x^u-\lambda
x^v$ maps to 
a unit under $\kk[x]/I \to (\kk[x]/I)_{P_\ell}$, because $u-v \notin
\Sat(L_\rho)$ and therefore
$x^u-\lambda x^v \in \kk[\NN^\Delta] \minus I(\chi_\ell)$. We conclude
$m\in Q_\ell$, as we wanted.
%
\end{proof}

In the previous theorem, the monomial ideal added to obtain a primary
component depends on the associated prime. However, this should not be
the case. If we assume that $\chr(\kk)=0$ and use~\cite{DMM} to
compute primary components, the monomial ideal there depends only on
the lattice of the associated prime, and not really on the partial
character. (See also Remark~\ref{rmk:connectionToDMM}.)
Motivated by this evidence, we provide an improvement to Theorem~\ref{thm:cellularPrimDecV1}.

\begin{theorem}
\label{thm:cellularPrimDecV2}
Under the same assumptions and notation as in
Theorem~\ref{thm:cellularPrimDecV1},
the $P_\ell$-primary component of $I$ is
\[
\bigg( (I + I(\rho_\ell)) : (\prod_{i \in \Delta} x_i)^\infty \bigg) + \Memb(I).
\]
\end{theorem}

\begin{proof}
We use the same notation as in the proof of
Theorem~\ref{thm:cellularPrimDecV1}, so that $J_\ell =
\big( (I + I(\rho_\ell)) : (\prod_{i \in \Delta} x_i)^\infty
\big)$. We wish to show that 
\begin{equation}
\label{eqn:componentMembI}
J_\ell + \Memb(J_\ell) = J_\ell + \Memb(I)
\end{equation}

We first claim that $\Memb(I) \subseteq J_\ell + \Memb(J_\ell)$ (which
implies
$\supseteq$ in~\eqref{eqn:componentMembI}).
As in the
proof of Theorem~\ref{thm:cellularPrimDecV1}, the lattice
corresponding to $J_\ell \cap
\kk[\NN^\Delta]$ has saturation
$\Sat(L_\rho)$. Therefore, if $m \in \Memb(I) \cap \kk[\NN^\bDelta]$ is
a monomial,
either $m \in J_\ell$, or
the binomial produced by Lemma~\ref{lemma:monomialsInMemb} for $m$ and
$I$ can be used to show that $m \in \Memb(J_\ell)$, since $I
\subseteq J_\ell$.

Now~\eqref{eqn:componentMembI} follows if we show that $\Memb(J_\ell) \subseteq \Memb(I)$.

Let $x^\mu$ be a generator of
$\Memb(J_\ell)$, that is, a witness monomial to an embedded
associated lattice of $J_\ell$, so in particular, $x^\mu
\in \kk[\NN^\bDelta]$, and $x^\mu \notin J_\ell$. Pick a binomial 
$x^u-\lambda x^v \in \kk[\NN^\Delta]$ as in
Lemma~\ref{lemma:monomialsInMemb} for $J_\ell$ and $x^\mu$. After multiplying by
a monomial in $\kk[\NN^\Delta]$, we may assume that $x^\mu(x^u-\lambda
x^v) \in I + I(\rho_\ell)$. As $x^\mu \notin J_\ell$, $x^\mu x^u$ and 
$x^\mu x^v$ do not belong to $I + I(\rho_\ell)$. We write the binomial
$x^\mu(x^u-\lambda x^v)$ as a linear combination of binomials and
monomials in $I$, and binomials of the form $mb$ where $m \in \kk[x]$
is a monomial and $b \in \kk[\NN^\Delta]$ is a generator of $I(\rho_\ell)$. 
%
%
%
%
%
We visualize this expression as a graph
$G$ with possibly multiple edges, whose vertices are the exponent
vectors of all the terms involved.
Two vertices are joined by an edge if 
the corresponding monomials are part of one of the binomials in the
combination. Note that an 
edge $(\alpha,\gamma)$ arising from an element of $I(\rho_{\ell})$
satisfies $\alpha_\bDelta = \gamma_\bDelta$, where 
$\alpha_\bDelta$ is the element of $\ZZ^\bDelta$ whose coordinates
indexed by $\bDelta$ coincide with those of $\alpha$, and similarly
for $\gamma_\bDelta$.

It is not hard to show that, since $x^\mu x^u$ and 
$x^\mu x^v$ do not belong to $I+I(\rho_\ell)$, the
vertices $\mu+u$ and $\mu+v$ belong to the same connected
component of $G$. Therefore there is a path $\Gamma$ in $G$ connecting
$\mu+u$ to $\mu+v$, that is, there exists a sequence of edges
$\varepsilon_1=
(\alpha_1,\gamma_1),\dots,\varepsilon_t=(\alpha_t,\gamma_t)$ of $G$
such that
$\alpha_1=\mu+u$, $\gamma_i = \alpha_{i+1}$ for $i=1,\dots,t-1$
and $\gamma_t =\mu+v$.

%

Suppose that one of the edges $\varepsilon_j = (\alpha,\gamma)$ 
(we remove the indices from the vertices to simplify notation) of the path $\Gamma$ is
such that $\alpha_\bDelta = \gamma_\bDelta$ and $\gamma-\alpha
\notin \Sat(L_\rho)$. This last condition implies that the edge $\varepsilon_j$
arises from a binomial in $I$.
By Lemma~\ref{lemma:monomialsInMemb}, we see that
$x^{\alpha_\bDelta}=x^{\gamma_\bDelta} \in I + \Memb(I)$.
Now consider $\varepsilon_{j-1} = (\hat{\alpha},\hat{\gamma})$, so that
$\hat{\gamma}=\alpha$. Either $\hat{\alpha}_\bDelta=\hat{\gamma}_\bDelta
(=\alpha_\bDelta)$, in which case $x^{\hat{\alpha}_\bDelta} \in I +
\Memb(I)$, or $\hat{\alpha}_\bDelta\neq\hat{\gamma}_\bDelta$, so that the
binomial corresponding to $\varepsilon_j$ belongs to $I$, and
consequently $x^{\hat{\alpha}_\bDelta} \in I + \Memb(I)$ again. Propagating
this argument along $\Gamma$, we see that $x^\mu \in I+\Memb(I)$, and
since $x^\mu \notin I$, we have $x^\mu \in \Memb(I)$.

%

We may now assume that for every edge
$\varepsilon_j=(\alpha_j,\gamma_j)$ of $\Gamma$ such that $(\alpha_j)_{\bDelta} = 
(\gamma_j)_\bDelta$, we have $\alpha_j-\gamma_j \in \Sat(L_\rho)$. 

Let $\varepsilon_{i_1},\dots,\varepsilon_{i_q}$ be the subsequence of
$\varepsilon_1,\dots,\varepsilon_t$ consisting of edges
$\varepsilon_j$ such that $(\alpha_j)_{\bDelta} \neq
(\gamma_j)_\bDelta$, with the convention that
$i_1<i_2<\cdots<i_q$. Note that the previous assumption and $u-v\notin
\Sat(L_\rho)$ imply that $q \geq 1$.
Each of the edges $\varepsilon_{i_j}$ is 
associated to a binomial that lies in $I$. 
By construction, $(\alpha_{i_1})_\bDelta = \mu = (\gamma_{i_q})_\bDelta$, and $(\gamma_{i_j})_\bDelta=
(\alpha_{i_{j+1}})_\bDelta$ for $i=1,\dots,q-1$.

If $q=1$, we can use the binomial corresponding to $\varepsilon_{i_1}$
to show that $x^\mu \in \Memb(I)$, so assume $q \geq 2$.
Write $\varepsilon_{i_1} =(\alpha,\gamma)$ and
$\varepsilon_{i_2}=(\hat{\alpha},\hat{\gamma})$, and consider the corresponding
binomials $f$ and $\hat{f}$ respectively. As
$\gamma_\bDelta=\hat{\alpha}_\bDelta$, we can choose monomials
$m,\hat{m} \in \kk[\NN^\Delta]$ and $\kappa,\hat{\kappa} \in \kk$ 
such that the \emph{S-pair} $g = \kappa m f - \hat{\kappa} \hat{m} \hat{f}$
cancels the $x^\gamma$ term in $f$ with the $x^{\hat{\alpha}}$ in
$\hat{f}$. Since $f,\hat{f} \in I$, $g \in I$.
The binomial $g$ is a linear combination of two monomials $x^{\bar{\alpha}}$
and $x^{\bar{\gamma}}$ satisfying $\bar{\alpha}_\bDelta = \mu$,
$\bar{\gamma}_\bDelta= \hat{\gamma}_\bDelta$, and
$\bar{\alpha}-\bar{\gamma}$ is congruent to $\alpha-\hat{\gamma}$
modulo $\Sat(L_\rho)$. We repeat this procedure with $g$ and the binomial
corresponding to $\varepsilon_{i_3}$ and continue, until
$\varepsilon_{i_q}$ is used. The outcome is a binomial in $I$ that can be
used, together with Lemma~\ref{lemma:monomialsInMemb}, to show that $x^\mu \in \Memb(I)$.
\end{proof}

\begin{remark}
\label{rmk:connectionToDMM}
When $\chr(\kk) = 0$,~\cite[Theorem~3.2.1]{DMM} can be applied to the
situation of Theorem~\ref{thm:cellularPrimDecV2}. The expression for 
the primary component given in~\cite{DMM} looks different
from~\eqref{eqn:componentMembI}, but it is not hard to show directly
that the two do indeed coincide, as they must, since minimal primary
components are uniquely determined.

We wish to emphasize that while~\cite[Theorem~3.2]{DMM} does not need
a cellularity hypothesis on the binomial ideals considered, it does 
require the base field to have characteristic zero. Moreover,
the proof of that result is based on the fact that the primary
decomposition of a binomial ideal over a field of characteristic zero
can be reduced to considering the primary components of binomial
ideals in monoid rings associated to prime ideals arising from faces
of the monoid, a fact that does not hold when $\chr(\kk) > 0$.

On the other hand, although Theorem~\ref{thm:cellularPrimDecV1}
applies regardless of the characteristic of $\kk$, the cellularity
assumption is crucial in our arguments.
\endrk
\end{remark}

The following result provides an explicit computation for the hull of
a cellular binomial ideal, and in particular, gives an alternative
proof for the fact that such a hull is binomial (see~\cite[Corollary~6.5]{ES}).

\begin{theorem} 
\label{thm:cellularHull}  
If $I$ is a $\Delta$-cellular binomial ideal then 
$\hull(I) = I + \Memb(I)$.
\end{theorem}

\begin{proof}
We already know from Theorem~\ref{thm:preliminaryToHull} that
$I+\Memb(I)$ is $\Delta$-cellular, unmixed, and with the same minimal
associated primes as $I$. 

Assume first that $\kk$ is algebraically closed.
We use the same notation as in the proof of
Theorems~\ref{thm:cellularPrimDecV1} and~\ref{thm:cellularPrimDecV2};
our goal is to show that the primary
component of $I+\Memb(I)$ corresponding to the (minimal) associated prime
$P_\ell = \kk[x] \cdot I(\chi_\ell) + \< x_i \mid i \in \bDelta\>$
equals the $P_\ell$-primary component of $I$.

By Theorems~\ref{thm:cellularPrimDecV2} and~\ref{thm:preliminaryToHull}, 
the $P_\ell$-primary component of $I+\Memb(I)$ is 
\[
\big((I+\Memb(I)+I(\rho)_\ell): (\prod_{i \in \Delta}x_i)^\infty\big),
\]
which clearly contains the $P_\ell$-primary component,
$\big( (I+I(\rho_\ell)):(\prod_{i\in
  \Delta}x_i)^\infty \big) + \Memb(I)$, of $I$. As this ideal is $\Delta$-cellular,
%
%
\[
\big( (I+I(\rho_\ell)):(\prod_{i\in \Delta}x_i)^\infty \big) + \Memb(I) = 
\bigg( \bigg[
\big( (I+I(\rho_\ell)):(\prod_{i\in \Delta}x_i)^\infty \big) + \Memb(I)
\bigg] : (\prod_{i\in \Delta}x_i)^\infty \bigg).
\]
Therefore
$\big((I+\Memb(I)+I(\rho)_\ell): (\prod_{i \in
  \Delta}x_i)^\infty\big)$ is contained in 
$\big( (I+I(\rho_\ell)):(\prod_{i\in \Delta}x_i)^\infty \big) +
\Memb(I)$, which shows that the $P_\ell$-primary components of $I$ and
$I+\Memb(I)$ coincide.

Now assume that $\kk$ is not algebraically closed. Write $\bar{I}$ for
the extension of $I$ to $\bar{\kk}[x]$. Let $P$ be a minimal prime of
$I$, and $\bar{P}$ a minimal prime of $\bar{I}$ lying over $P$.
If $\bar{Q}$ is the $\bar{P}$-primary
component of $\bar{I}$, then $Q = \bar{Q} \cap \kk[x] \supseteq I$ is $P$-primary
(apply Theorem~\ref{thm:associatedPrimesFieldExtension} to see that
$P$ is the unique associated prime of $Q$). This implies that $Q$
contains the $P$-primary component of $I$
by~\cite[Corollary~10.21]{AM}. Consequently $\hull(I) \subseteq
\hull(\bar{I}) \cap \kk[x]$. But we know that $\hull(\bar{I}) =
\bar{I} + \Memb(\bar{I})$. Since $\Memb(\bar{I})=\bar{\kk}[x] \cdot \Memb(I)$, we see
that $I + \Memb(I) = \hull(\bar{I}) \cap \kk[x] \supseteq \hull(I)$.

Clearly, $I \subseteq \hull(I)$, so in order to prove $I+\Memb(I) =
\hull(I)$ it is enough to show that $\Memb(I) \subseteq \hull(I)$. Let
$m \in \kk[\NN^{\bDelta}]$ be a generator of $\Memb(I)$, and let $b = x^u
-\lambda x^v \in \kk[\NN^\Delta]$ be the binomial produced by
Lemma~\ref{lemma:monomialsInMemb}. Then $b$ does not belong to any
minimal prime of $I$, since it does not belong to any minimal prime of
$\bar{I}$ (use the fact that $u-v \notin \Sat(L_\rho)$). This and $mb
\in I$ imply that $m$ belongs to the kernel of every 
localization map $(\kk[x]/I) \to (\kk[x]/I)_P$, for $P$ a minimal
prime of $I$.
\end{proof}

\section{An unmixed decomposition for a cellular binomial ideal}
\label{sec:unmixed}
 
In this section we complete the proof of
Theorem~\ref{thm:unmixedDecomposition} by constructing an 
unmixed decomposition for a cellular binomial ideal. We work by
Noetherian induction; the following result provides the inductive step.

\begin{theorem}
\label{thm:inductionStep}
Let $I$ be a $\Delta$-cellular binomial ideal in $\kk[x]$. 
Let $(L_{\tau_1},\tau_1),\dots,(L_{\tau_k},\tau_k)$ be partial characters
on $\ZZ^\Delta$ corresponding to embedded associated lattices of $I$
(so in particular there exist monomials $m_j \in \Memb(I)$
with $(I:m_j)= I(\tau_j)$) such that the ideals $I(\tau_1),\dots,I(\tau_k)$
are minimal (with respect to inclusion) among the lattice ideals
arising in this way. (Call such lattice ideals the \emph{minimal
  embedded lattice ideals of $I$}.)
Then 
\begin{equation}
\label{eqn:inductionStep}
I  = \bigg[ I + \Memb(I) \bigg] \bigcap \bigg[\bigcap_{j=1}^k \big( (I
+ I(\tau_j)): (\prod_{i\in \Delta} x_i)^\infty \big) \bigg] .
\end{equation}
\end{theorem}

\begin{proof}
By induction on $k$. The base case is $k=0$, when $I$ has no
embedded associated lattices, $\Memb(I)=0$, and $I$ is unmixed, so
that~\eqref{eqn:inductionStep} clearly holds.

Now let $k \geq 1$. Set $M = \< m \in \Memb(I) \mid (I:m) \cap
\kk[\NN^\Delta] = I(\tau) \text{ with } I(\tau) \supseteq I(\tau_1)
\>$. We claim that $I = (I+M) \cap (I+I(\tau_1))$. 

To see this, we first observe that the monomials in $I+M$ belong to either $I$ or
$M$, which follows by the same proof as Proposition~\ref{propo:monomialsInI+Memb(I)}.

Let $f \in (I+M) \cap (I+I(\tau_1))$ and use $f \in I+I(\tau_1)$ to
write $f = f_1+f_2+f_3$, where $f_1 \in I$, $f_2 \in \kk[x] \cdot
I(\tau_1)$ is a linear combination of binomials in $I(\tau_1)$ times
monomials that are either in $I$ or in $M$, and $f_3 \in \kk[x] \cdot
I(\tau_1)$ is a linear combination of binomials in $I(\tau_1)$ times
monomials that are neither in $I$ nor in $M$. Note that, by definition
of $M$, the product of
a monomial in $M$ times a binomial in $I(\tau_1)$ belongs to $I$.
Thus $f_3 = f - f_1 -f_2 \in I + M$, and its individual
terms do not belong to $I+M$. By
Proposition~\ref{propo:binomial+monomial}, $f_3 \in I$, and
consequently $f \in I$. This shows that $(I+M) \cap (I+I(\tau_1))
\subseteq I$, and the other inclusion is obvious.

As $I$ is $\Delta$-cellular, we can saturate both sides of $I = (I+M)
\cap (I+I(\tau_1))$ to obtain $I =
\big((I+M):(\prod_{i\in\Delta}x_i)^\infty\big) \cap \big(
(I+I(\tau_1)): (\prod_{i\in\Delta}x_i)^\infty \big)$. We wish to apply
the inductive hypothesis to the $\Delta$-cellular ideal
$\big((I+M):(\prod_{i\in\Delta}x_i)^\infty\big)$, but in order to do
this, we need to show that this ideal has fewer minimal
embedded lattice ideals than $I$.

Let $m \in \kk[\NN^\bDelta]$ be a monomial not in $I+M$. If $b \in
\kk[\NN^\Delta]$ is a binomial such that $mb \in
\big((I+M):(\prod_{i\in\Delta}x_i)^\infty\big)$, then there exists a
monomial $m' \in \kk[\NN^\Delta]$ such that $m'm b \in I+M$. The
individual terms of the binomial $m'm b$ do not belong to $I+M$,
since $m$ does not (and the ideal of all monomials in $I+M$ is generated
by monomials in $\kk[\NN^\bDelta]$). 
Thus,
Proposition~\ref{propo:binomial+monomial} implies that $m'm b$
belongs to the $\Delta$-cellular ideal $I$, whence $mb \in
I$. 
This means that $\big(
\big((I+M):(\prod_{i\in\Delta}x_i)^\infty\big): m \big) \cap
\kk[\NN^\Delta]= (I:m)\cap \kk[\NN^\Delta]$. Since $\big(
\big((I+M):(\prod_{i\in\Delta}x_i)^\infty\big): m \big) = \langle 1 \rangle$
if $m \in I+M$, we conclude that the minimal embedded
lattice ideals of
$\big((I+M):(\prod_{i\in\Delta}x_i)^\infty\big)$ are $I(\tau_2),\dots,I(\tau_k)$.

By inductive hypothesis, 
\begin{align*}
\big((I+M):(\prod_{i\in\Delta}x_i)^\infty\big) = &
\bigg[
\big((I+M):(\prod_{i\in\Delta}x_i)^\infty\big) + \Memb
\big((I+M):(\prod_{i\in\Delta}x_i)^\infty\big) \bigg] \\
& \bigcap 
\bigg[
\bigcap_{j=2}^k \bigg( \big[ \big((I+M):(\prod_{i\in\Delta}x_i)^\infty\big)
+ I(\tau_j)\big] : (\prod_{i\in\Delta}x_i)^\infty \bigg)
\bigg] 
\end{align*}
and therefore 
\begin{align}
\nonumber
I  = &
 \bigg[
\big((I+M):(\prod_{i\in\Delta}x_i)^\infty\big) + \Memb
\big((I+M):(\prod_{i\in\Delta}x_i)^\infty\big) \bigg] \\
\nonumber
& \bigcap 
\bigg[
\bigcap_{j=2}^k \bigg( \big[ \big((I+M):(\prod_{i\in\Delta}x_i)^\infty\big)
+ I(\tau_j)\big] : (\prod_{i\in\Delta}x_i)^\infty \bigg)
\bigg] \\
& \bigcap
\big( (I+I(\tau_1)) : (\prod_{i\in\Delta}x_i)^\infty\big).
\label{eqn:largerIntersection}
\end{align}
Every term of the right hand side above contains the
corresponding term of 
\[ 
(I + \Memb(I)) \bigcap \big[ \bigcap_{j=2}^k  \big( (I+I(\tau_j)) : (\prod_{i\in\Delta}x_i)^\infty\big)\big] \bigcap \big( (I+I(\tau_1)) : (\prod_{i\in\Delta}x_i)^\infty\big);
\]
that the first intersectand in~\eqref{eqn:largerIntersection} contains
$I+\Memb(I)$ follows by the same argument that showed that the minimal embedded
lattice ideals of
$\big((I+M):(\prod_{i\in\Delta}x_i)^\infty\big)$ are $I(\tau_2),\dots,I(\tau_k)$.

We conclude that
\[
I  \supseteq \bigg[ I + \Memb(I) \bigg] \bigcap \bigg[\bigcap_{j=1}^k \big( (I
+ I(\tau_j)): (\prod_{i\in \Delta} x_i)^\infty \big) \bigg],
\]
and the reverse inclusion is clear.
\end{proof}

\begin{proof}[Proof of Theorem~\ref{thm:unmixedDecomposition}]
By Theorems~\ref{thm:cellularPrimDecV2} and~\ref{thm:cellularHull},
the only statement that remains to be shown is that $I$ equals
\begin{equation}
\label{eqn:unmixedDecompositionRHS}
\hull(I) \bigcap \bigg[ \bigcap_{m \in \Memb(I)}
\hull\bigg((I+I(\rho(m))):(\prod_{i\in\Delta} x_i)^\infty \bigg) \bigg], 
\end{equation}
where, for $m\in \Memb(I)$, $(I:m) \cap \kk[\NN^\Delta] = I(\rho(m))$.

If $I$ is unmixed, then $\Memb(I) = 0$, and the desired statement holds. 
By Noetherian induction, we may assume that the
statement holds for any
$\Delta$-cellular binomial ideal strictly containing $I$. 

Now consider the case that $I$ is mixed, and apply
Theorem~\ref{thm:inductionStep} to $I$. Each ideal $K_j = \big(
(I+I(\tau_j)) : (\prod_{i\in\Delta} x_i)^\infty \big)$ appearing in
Theorem~\ref{thm:inductionStep} is a $\Delta$-cellular binomial ideal
strictly containing $I$, so that 
\begin{equation}
\label{eqn:jthDecomposition}
K_j = \hull(K_j) \bigcap \bigg[ \bigcap_{m \in \Memb(K_j)}
\hull\bigg( (K_j+I(\tau_j(m))) : (\prod_{i\in\Delta} x_i)^\infty \bigg)\bigg], 
\end{equation}
where, for $m\in \Memb(K_j)$, $(K_j:m) \cap \kk[\NN^\Delta] = I(\tau_j(m))$.

Using the proof of Theorem~\ref{thm:cellularPrimDecV2}, we see that
$\Memb(K_j) \subseteq \Memb(I)$. Moreover, if $m \in \Memb(K_j)$, $K_j
\supset I$ implies that $I(\tau_j(m)) \supset I(\rho(m))$ and
$\hull\bigg( (K_j + I(\tau_j(m))):(\prod_{i\in\Delta}x_i)^\infty\bigg) \supseteq \hull \bigg( (I+I(\rho(m))):(\prod_{i\in\Delta}x_i)^\infty\bigg)$.

Thus, intersecting $\hull(I)$ and the
decompositions~\eqref{eqn:jthDecomposition} for $j=1,\dots,k$, we
obtain an ideal that contains the
intersection~\eqref{eqn:unmixedDecompositionRHS}. But this ideal equals
$I$ by construction. Moreover, it is clear that $I$ is contained
in~\eqref{eqn:unmixedDecompositionRHS}, and therefore it equals that
intersection, as we wanted.
\end{proof}

We conclude with three examples. The first one illustrates the mechanics
of the proof of Theorem~\ref{thm:unmixedDecomposition}; the second and
third ones show that the unmixed decomposition and
cellular binomial primary decomposition produced in
that theorem may both be redundant.

\begin{example}
\label{example:redundancies}
We are grateful to Christopher O'Neill, who shared this example with
us. 

Let $I = \<x_5(x_1-x_2), x_6(x_3-x_4), x_5^2, x_6^2,
x_5x_6\>$. Then $I$ is $\{1,2,3,4\}$-cellular, and
$\<x_5,x_6\>$ is the unique minimal
prime of $I$. In this case,
$\Memb(I) = \<x_5, x_6\>$, and the embedded associated lattice ideals
of $I$ arise via
\[
(I:x_5)\cap\kk[x_1,\dots,x_4] = \<x_1-x_2\>; \quad(I:x_6)\cap
\kk[x_1,\dots,x_4]=\<x_3-x_4\>. 
\]
The unmixed decomposition of $I$ given by
Theorem~\ref{thm:unmixedDecomposition} is 
\[
I = \<x_5,x_6\> \cap \< x_1-x_2, x_5^2,x_6\> \cap \<
x_3-x_4,x_5,x_6^2\>.
\]
On the other hand, if we follow the proof of that theorem, we first apply
Theorem~\ref{thm:inductionStep} to obtain
\begin{multline*}
I = \<x_5,x_6\> \cap \< x_1-x_2,
x_6(x_3-x_4), x_5^2, x_6^2, x_5x_6\>
 \cap \<x_5(x_1-x_2),x_3-x_4, x_5^2, x_6^2, x_5x_6\>.
\end{multline*}
In this case, the latter two intersectands have an embedded associated
lattice, so the proof of Theorem~\ref{thm:unmixedDecomposition}
requires us to perform
\begin{multline*}
\< x_1-x_2, x_6(x_3-x_4), x_5^2, x_6^2, x_5x_6\> =
\<x_1-x_2,x_5^2,x_6\> \cap \<x_1-x_2,x_3-x_4, x_5^2, x_6^2, x_5x_6\>,
\end{multline*}
and
\begin{multline*}
\<x_5(x_1-x_2),x_3-x_4, x_5^2, x_6^2,
x_5x_6\>=\<x_3-x_4,x_5,x_6^2\> \cap \<x_1-x_2,x_3-x_4,x_5^2, x_6^2, x_5x_6\>,
\end{multline*}
yielding the redundant intersectand $\<x_1-x_2,x_3-x_4,x_5^2, x_6^2, x_5x_6\>$.
\endrk
\end{example}

\begin{example}
\label{ex:redundantUnmixedDec}
Let $I = \<x_3^2(x_1^2-x_2^2), x_3(x_1^4-x_2^4), x_3^3\>$, a
$\{1,2\}$-cellular binomial ideal in
$\kk[x_1,x_2,x_3]$. In this case, the unmixed decomposition from
Theorem~\ref{thm:unmixedDecomposition} is
\[ 
I = \<x_3\> \cap
\<x_3^2(x_1^2-x_2^2),x_1^4-x_2^4,x_3^3\> \cap \<x_1^2-x_2^2,x_3^3\>,
\]
which is redundant, as the third intersectand contains the second.
\endrk
\end{example}

\begin{example}
\label{ex:redundantPrimDec}
Let $I = \<x_2(x_1^2-1), x_3(x_1^3-1), x_2^2,x_3^2,x_2x_3 \>$, a
$\{1\}$-cellular binomial ideal in
$\kk[x_1,x_2,x_3]$. In this case, the unmixed decomposition from
Theorem~\ref{thm:unmixedDecomposition} is
\begin{multline*}
I = \<x_2,x_3\> \cap
\<x_2(x_1^2-1),x_1^3-1,x_2^2,x_3^2,x_2x_3\> \cap \<x_1^2-1,x_3(x_1^3-1),x_2^2,x_3^2,x_2x_3\>.
\end{multline*}
Note that $P=\<x_1-1,x_2,x_3\>$ is an associated prime of both the
second and third intersectands. If $\chr(\kk) \neq 2,3$, their
corresponding $P$- primary components
both equal $\<x_1-1,x_2^2,x_3^2,x_2 x_3\>$.

If $\chr(\kk) =2$, we obtain $\<x_1^2-1,x_2^2,x_3^2,x_2x_3\>$
for the $P$-primary component of the second intersectand, and
$\<x_1-1,x_2^2,x_3^2,x_2x_3\>$ for the $P$-primary component
of the third.

If $\chr(\kk)=3$, we obtain $\<x_1-1,x_2^2,x_3^2,x_2x_3\>$ for
the $P$-primary component of the second intersectand, and
$\<x_1^3-1,x_2^2,x_3^2,x_2x_3\>$ for the $P$-primary component
of the third.

These calculations show that the primary
decomposition from Theorem~\ref{thm:unmixedDecomposition} is redundant
in this case. 
\endrk
\end{example}

\section{Computational Implications}
\label{sec:implications}

In this section we explore the implications of our main results,
especially to computational binomial primary decomposition. 
Considerations of possible computer implementation have informed all the
developments in this article. It is for this reason that we have tried to only
make assumptions on the base field when strictly necessary; many of
our proofs would be simplified if we required the field $\kk$ to be
algebraically closed throughout.

The \texttt{Macaulay2}~\cite{M2} package \texttt{Binomials}, implemented by Thomas
Kahle~\cite{K,KM2}, computes binomial primary decomposition using
Kahle's improvements of the algorithms in~\cite{ES,OS}. The input is
required to be an ideal generated by differences of monomials (a
\emph{unital} or \emph{pure difference} binomial ideal), and the base field is assumed to be
algebraically closed of characteristic zero.

We briefly describe this procedure, whose first step is
to find a cellular decomposition of the input binomial
ideal. For each cellular (binomial) component $J$, the ideal
$\Memb(J)$ is computed, keeping track of which lattice ideals occur,
so that the associated primes of $J$ may be found by saturating partial
characters. If $P = \kk[x] \cdot I(\chi) + \< x_i \mid i \in
\bDelta\>$ is such an associated prime, then the package
\texttt{Binomials} outputs 
$((J+I(\chi)):(\prod_{i \in \Delta} x_i)^\infty) + 
\Memb ((J+I(\chi)):(\prod_{i \in \Delta} x_i)^\infty)$ as the
$P$-primary component of $J$. 
The final step in the computation is to remove redundancies that occur
if an associated prime of a cellular component of the input ideal $I$
is not an associated prime of $I$.

We believe that our results can be used to facilitate the implementation
of a positive-characteristic version of the above
computation. 
First, we observe that the cellular decomposition of a binomial ideal
is characteristic independent, so we concentrate on the primary
decomposition of a cellular component $J$ of the binomial ideal $I$. 
The computation of $\Memb(J)$ is already implemented in
\texttt{Binomials}; the only modification needed to make it applicable
to any cellular binomial ideal over any field, is that a dimension
check needs to be added to verify the condition on saturations. 

The monomial ideal $\Memb(J)$ and all of the corresponding associated lattices need to
be computed in order to find ${\rm Ass}(\kk[x]/J)$. Since this
computation already contains most of the information necessary to
perform the unmixed decomposition~\eqref{eqn:unmixedDecompositionExplicit},
we propose computing this decomposition as an intermediate step.
Note that the computations of $\Memb(J)$ and of the unmixed
decomposition~\eqref{eqn:unmixedDecompositionExplicit} do not require
any assumptions on the base field, and therefore can be
performed over finite fields.

Next, the primary decomposition of an unmixed cellular binomial
ideal is given in Theorem~\ref{thm:hullDecomposition} (or
Theorem~\ref{thm:cellularPrimDecV2}), and requires a finite extension
of the base field in order to find saturations of partial
characters. The final step would be, as before, to remove redundancies.

We have thus outlined an adaptation of the current software that would
implement the 
computation of binomial primary decomposition of (unital) binomial
ideals over finite fields.

We finish this section with two results that can be quickly derived
from the statements in this article.
The first is an easy characterization of binomial primary
ideals; this statement seems clear from the construction of $\Memb(I)$
for a cellular binomial ideal $I$ and from
Theorem~\ref{thm:cellularWitnesses}, but it does require proof 
(especially since we do not wish to assume that the field $\kk$
is algebraically closed).

\begin{lemma}
\label{lemma:primaryBinomialIdeal}
Let $I$ be a binomial ideal in $\kk[x]$. Then $I$ is primary if and
only if $I$ is $\Delta$-cellular for some $\Delta \subseteq
\{1,\dots,n\}$, the lattice ideal $I \cap \kk[\NN^\Delta]$ is primary, and
$\Memb(I) =0$. 
\end{lemma}

\begin{proof}
If $I$ is primary then it is unmixed and $\Delta$-cellular for some $\Delta$, so by
Corollary~\ref{coro:zeroMemb}, we have $\Memb(I)=0$. Lattice ideals have no
embedded primes; thus, if $I \cap \kk[\NN^\Delta]$ were not primary, it would 
have more than one minimal prime, and therefore so would $I$ by
Lemma~\ref{lemma:minimalPrimesOfCellular}. 
The converse is Corollary~\ref{coro:primaryCellular}. 
\end{proof}

The final result in this section combines our computation of the hull
of a cellular 
binomial ideal, Theorem~\ref{thm:cellularHull}, with~\cite[Theorem~7.1']{ES}, one of the core
statements in that article, and its improvement~\cite[Theorem~3.2]{OS}.

If $q$ is a positive integer and $b = t_1-t_2$ is a binomial, where $t_1, t_2$ are
terms (that is, products of constants times monomials), set
$b^{[q]} = t_1^q - t_2^q$. Note
that this binomial is well-defined only up to a constant multiple, as changing the order
of the terms of $b$ makes a difference. 
If $b$ has only one term, set $b^{[q]} = b^q$, its ordinary $q$th power.
For $I$ a binomial ideal, we define its \emph{$q$th quasipower} to be the ideal 
$I^{[q]} = \< b^{[q]} \mid b \in I \text{ is a binomial}\>$.

\begin{corollary}
\label{coro:noHullInESTheorem}
 Let $I$ be a binomial ideal in $\kk[x]$, where $\kk$ is algebraically
  closed. If $P$ is an associated prime of $I$, write $\Delta(P)$ for
  the set of nonzerodivisor variables of $P$.
\begin{enumerate}[leftmargin=*]
\item If $\chark = p > 0$, and $q=p^e$ is sufficiently large, a
  minimal primary decomposition of $I$ into binomial ideals is given by
\[
I  = \bigcap_{P \in {\rm Ass}(\kk[x]/I)} \bigg[ \bigg( (I + P^{[q]}): (\prod_{i\in
  \Delta(P)} x_i)^\infty \bigg) + \Memb \bigg( (I + P^{[q]}): (\prod_{i\in
  \Delta(P)} x_i)^\infty \bigg) \bigg].
\]
\item If $\chark = 0$,  a
  minimal primary decomposition of $I$ into binomial ideals is given by
\begin{equation*}
\begin{split}
\pushQED{\qed}
I  = \bigcap_{P \in {\rm Ass}(\kk[x]/I)} 
\bigg[ \bigg( \big(
I   + (P \cap & \kk[\NN^{\Delta(P)}]) \big)  :
(\prod_{i \in \Delta(P)} x_i )^\infty \bigg) + \\
& \Memb \bigg( \big(
I + (P \cap \kk[\NN^{\Delta(P)}]) \big) : 
(\prod_{i \in \Delta(P)}x_i)^\infty \bigg)
\bigg].  \qedhere
\popQED
\end{split}
\end{equation*}
\end{enumerate}
\end{corollary}


\end{document}